\documentclass{amsart}
\usepackage[T1]{fontenc}
\usepackage[utf8]{inputenc}
\usepackage{lmodern}
\usepackage[a4paper]{geometry}
\usepackage{amsthm}
\usepackage{amsmath}
\usepackage{amssymb}
\newcommand{\lkul}{\Lambda _{\textrm{Kul}}}
\newcommand{\lmyr}{\Lambda _{\textrm{Myr}}}
\newcommand{\R}{\mathbb{R}}
\newcommand{\Hbb}{\mathbb{H}}
\newcommand{\sol}{\mathbb{S}\mbox{ol}}
\newcommand{\polyDisc}{\mathbb{H}\times\mathbb{H}}
\newcommand{\C}{\mathbb{C}}
\newcommand{\Proy}{\mathbb P_{\mathbb C}}
\newcommand{\heisenberg}{\mathrm{Heis}}
\newcommand{\solRep}{\begin{pmatrix}
		\lambda^t &     0        & a\,x + b\,y\\
		   0      & \lambda^{-t} & c\,x + d\, y\\
		   0      &     0        &   1
	\end{pmatrix}}

\renewcommand\l\left
\renewcommand\r\right

\newtheorem{theorem}{Theorem}[section]
\newtheorem{proposition}{Proposition}[section]

\newtheorem{corollary}[theorem]{Corollary}

\theoremstyle{definition}
\newtheorem{definition}[theorem]{Definition}
\newtheorem{example}[theorem]{Example}

\theoremstyle{remark}
\newtheorem{remark}[theorem]{Remark}

\numberwithin{equation}{section}

\begin{document}
	\title{A note  about the topological type of families of complex kleinians groups in $\Proy^2$}

	\author{Waldemar Barrera Vargas}
	\address{Facultad de matemáticas, Mérida, México}
	\curraddr{Facultad de matemáticas, Mérida, México}
	\email{bvargas@correo.uady.mx}
	
	\author{René García Lara}
	\address{Facultad de matemáticas, Mérida, México}
	\curraddr{Facultad de matemáticas, Mérida, México}
	\email{rene.garcia@correo.uady.mx}
	
	\author{Juan Navarrete Carrillo}
	\address{Facultad de matemáticas, Mérida, México}
	\curraddr{Facultad de matemáticas, Mérida, México}
	\email{jp.navarrete@correo.uady.mx}

	\subjclass[2000]{Primary }
	
	\date{}
	
	\begin{abstract}
	We give a complete description of  the topological type of the  quotient space $\Omega/G$  of Complex Kleinian Groups with the maximum numbers
	of complex projectives lines in general position contained in its Kulkarni's limit set is four.
	\end{abstract}
	
	\maketitle
	
\section{Introduction}
The complex Kleinian groups were introduced by Jos\'e Seade and Alberto Verjosky \cite{SV} in  order to study the discrete subgroups of the group
of automorphisms of complex projective spaces. These groups are natural generalizations of Kleinian groups  in the context of hyperbolic spaces. 
A complex Kleinian group $G$ is a subgroup of $\textrm{PSL}(n+1,\mathbb{C})$  for which exists a $G$-invariant open non empty set 
of $\Proy^n$ where the action of $G$ is properly discontinuous. We notice the group $G$ is a  discrete subgroup of 
$\textrm{PSL}(n+1,\mathbb{C})$  but the converse is not necessarily true, for example the group $\textrm{PSL}(3,\mathbb{Z})$ is a discrete
subgroup of $\textrm{PSL}(3,\mathbb{C})$ but this not  a complex Kleinian group \cite{BCN5} and that give a difference with the kleinian groups of
hyperbolic spaces. Another important difference with the classical Kleinian groups is the concept of  limit set,  in the  case
of complex Kleinian groups with  do not have an standard definition, we have  three notions: Kulkarni limit set, Myrberg limit set, complement
 of a maximal region of discontinuity which are discussed in detail in \cite{}, 
but by some additional hypotheses of all these concepts of limit set are equivalents \cite{BCN}. In \cite{CS} Angel Cano and Jos\'e Seade show that every infinite discrete
subgroup of $\textrm{PSL}(3,\mathbb{C})$ has a complex projective  line contained in   the  limit set, in consequence, the limit set of  infinite
subgroups of $\textrm{PSL}(3,\mathbb{C})$ is an uncountable subset of $\Proy^2$, this result give a  big difference with the  classical 
Kleinian groups because the most simple groups called elementary the cardinal of his limit set is 0, 1 or 2 but in the case of complex Kleinian groups
the cardinal of the limit set is always zero or infinite.

An interesting problem, in analogy with the classical theory of Kleinian groups, is to define we understand as complex elementary Kleinian group.  
An alternative is to define an elementary complex Kleinian group, as that group whose limit set contains a finite number of complex  projectives lines 
\cite{CNS}, this definition  is good in some cases, but we can construct examples of groups whose limit set contained infinitely many complex projectives 
lines but only a  finite number of them in general position, then we can define an elementary complex Kleinian group of type II of the last way.

In \cite{BCN4} the authors give  a caractherization of complex Kleinian groups  with the maximum numbers of complex projectives lines 
in general position contained in its Kulkarni's limit set is four. In this article we  described the topology of these groups a  we obtain the following
theorem

\begin{theorem}\label{main 1}
Let $G$ be complex Kleinian group  with the maximum numbers of complex projectives lines in general position containes in  it's Kulkarni limit set is four,
then we have  the  follows:
\begin{enumerate}
 \item The  group $G$ is isomorphic to a lattice of the group  $\sol$
 \item Let $\Omega_0$ be a $G$-invariant connected component of the Kulkarni  discontinuity  region of $G$, then
 $\Omega_0/G$ is diffeomorphic to $\left( \sol/G \right)\times \mathbb{R}$. 
\end{enumerate}

\end{theorem}

\begin{corollary}
 There is a countable number of complex Kleinian groups non isomorphic with the maximum  numbers of complex projectives lines in general position containes 
 in  it's Kulkarni limit set is four.
\end{corollary}

\begin{corollary}
 Under the hypotheses of Theorem \ref{main}, $\Omega_0/\Gamma$ is a  fiber bundle with base $\mathbb{S}^1\times \mathbb{R}$ and 
 fiber $\mathbb{T}^ 2\times \mathbb{R}$.
\end{corollary}

\begin{theorem}\label{main2}
 Let $G$ a lattice of  three dimensional real Heisenberg group $\mathcal{H}$, then there exist a $G$-invariant open set
 $\Omega\subset \mathbb{P}^2_{\mathbb{C}}$ such that $\Omega/G$ is diffeomorphic to $\left(\mathcal{H}/G\right) \times \mathbb{R}$.
\end{theorem}

This article is organized in the  following  way:
In  section 1 we give some basic preliminaries about complex Kleinian groups, and we give a description of the groups with maximum
numbers of complex lines contained its Kulkarni's limit  set equal four \cite{BCN4}.
The next  section we give the proof of Theorem \ref{main} and its corollaries, in fact the  strategy of  the proof is the following:

Firstly  we notice that $G$ is naturally identified with the lattice of $\sol$,  $\mathbb{Z}^2\rtimes_{A} \mathbb{Z}$, where 
$A\in \textrm{SL}(2,\mathbb{Z})$ is an hyperbolic automorphism of torus. In the second section we defined an immersion of the group 
$\sol$ to the bydisc $\mathbb{H}\times \mathbb{H}$ and we show this  immersion can be extended to a $G$-equivariant  diffeomorphism of
$\sol\times \mathbb{R}$ to $\mathbb{H}\times \mathbb{H}$ and we finish the proof.
The corollaries are  consequence of the Theorem \ref{main}  and proposition 30 \cite{H}.
For  the Theorem \ref{main2}, the procedure is similar to  Theorem \ref{main} except that we do not have a  diffeomorphism $G$-equivariant
 between $\mathbb{C}\times \mathbb{H}$ and $\mathcal{H}\times \mathbb{R}$, nevertheles because the  natural action of $G$ in 
 $\mathbb{C}\times \mathbb{H}$ translated to $\mathcal{H}\times \mathbb{R}$ is  the clasical action in the firt factor of $G$ on $\mathcal{H}$ and
 in the second factor is the identity, we can established the proof of this theorem.

\section{Preliminaries}
The purpose of this section is to provide some definitions and results about complex Kleinian groups that will be helpful to the reader, 
for more details see \cite{CNS}, \cite{BCN} and \cite{BCN4}.

\subsection{Projective Geometry}
We recall that the complex projective plane
$\mathbb{P}^2_{\mathbb{C}}$ is defined as the orbit space of the usual scalar multiplication action of Lie group $\mathbb{C}^{*}$  in 
$\mathbb{C}^{3}\setminus \{\mathbf{0}\}$ and it is denoted by
$$\mathbb{P}^2_\mathbb{C}:=(\mathbb{C}^{3}\setminus \{\mathbf{0}\})/\mathbb{C}^*,$$
 This  is   a  compact connected  complex $2$-dimensional manifold.
Let $[\mbox{ }]:\mathbb{C}^{3}\setminus\{\mathbf{0}\}\rightarrow
\mathbb{P}^{2}_{\mathbb{C}}$ be the quotient map. If
$\beta=\{e_1,e_2,e_3\}$ is the standard basis of $\mathbb{C}^3$, we
write $[e_j]=e_j$, for $j=1,2,3$, and if
$\mathbf{z}=(z_1,z_2,z_3)\in \mathbb{C}^3\setminus\{\mathbf{0}\}$
then we  write $[\mathbf{z}]=[z_1:z_2:z_3]$. Also, $\ell\subset
\mathbb{P}^2_{\mathbb{C}}$ is said to be a complex line if
$[\ell]^{-1}\cup \{\mathbf{0}\}$ is a complex linear subspace of
dimension $2$. Given two distinct points
$[\mathbf{z}],[\mathbf{w}]\in \mathbb{P}^2_{\mathbb{C}}$, there is a
unique complex projective line passing through $[\mathbf{z}]$ and
$[\mathbf{w}]$, such complex projective line is called a
\emph{line}, for short, and it is denoted by
$\overleftrightarrow{[\mathbf{z}],[\mathbf{w}]}$. Consider the
action of $\mathbb{C}^*$ on  $\textrm{GL}(3,\mathbb{C})$ given by the
usual scalar multiplication, then
$${\rm PGL}(3,\mathbb{C})=\textrm{GL}(3,\mathbb{C})/ \mathbb{C}^*$$ is a Lie group
whose elements are called projective transformations.  Let $[[\mbox{
}]]:\textrm{GL}(3,\mathbb{C})\rightarrow {\rm PGL}(3,\mathbb{C})$ be
the quotient map,   $g\in {\rm PGL}(3,\mathbb{C})$ and
$\mathbf{g}\in \textrm{GL}(3,\mathbb{C})$, we say that $\mathbf{g}$
is a  lift of $g$ if  $[[\mathbf{g}]]=g$. One can show that $ {\rm
PGL}(3,\mathbb{C})$ is a Lie group  which acts transitively,
effectively and by biholomorphisms on $\mathbb{P}^2_{\mathbb{C}}$ by
$[[\mathbf{g}]]([\mathbf{w}])=[\mathbf{g}(\mathbf{w})]$, where
$\mathbf{w}\in \mathbb{C}^3\setminus\{ \mathbf{0}\}$ and
$\mathbf{g}\in \textrm{GL}(3, \mathbb{C})$.

We could have considered the action of the cube roots of unity $\{1,
\omega , \omega ^2 \} \subset \mathbb{C}^*$ on $\textrm{SL}(3,
\mathbb{C})$ given by the usual scalar multiplication, then
$${\rm PSL}(3,\mathbb{C})=\textrm{SL}(3,\mathbb{C})/ \{1, \omega ,
\omega ^2\} \cong \textrm{PGL}(3, \mathbb{C}).$$

We denote by $\textrm{M}_{3 \times 3}(\mathbb{C})$ the space of all $3
\times 3$ matrices with entries in $\mathbb{C}$ equipped with the
standard topology. The quotient space
\[ \textrm{SP}(3, \mathbb{C}):=(\textrm{M}_{3 \times 3}(\mathbb{C}) \setminus \{\mathbf{0}\})/\mathbb{C} ^* \]
is called the space of \emph{pseudo-projective maps of} $\mathbb{P}
_{\mathbb{C}} ^2$ and it is naturally identified with the projective
space $\mathbb{P} _{\mathbb{C}} ^8$. Since $\textrm{GL}(3,\mathbb{C})$ is an
open, dense, $\mathbb{C} ^*$-invariant set of $\textrm{M}_{3 \times
3}(\mathbb{C}) \setminus \{\mathbf{0}\}$, we obtain that the space of
pseudo-projective maps of $\mathbb{P} _{\mathbb{C}} ^2$ is a
compactification of ${\rm PGL}(3, \mathbb{C})$ (or ${\rm PSL}(3,
\mathbb{C})$). As in the case of projective maps, if $\mathbf{s}$ is an
element in $\textrm{M}_{3 \times 3}(\mathbb{C}) \setminus
\{\mathbf{0}\}$, then $[\mathbf{s}]$ denotes the equivalence class
of the matrix $\mathbf{s}$ in the space of pseudo-projective maps of
$\mathbb{P} _{\mathbb{C}} ^2$. Also, we say that $\mathbf{s}\in
\textrm{M}_{3 \times 3}(\mathbb{C})\setminus \{\mathbf{0}\}$ is a lift
of the pseudo-projective map $S$ whenever $[\mathbf{s}]=S$.

Let $S$  be an element in $( \textrm{M}_{3 \times
3}(\mathbb{C})\setminus \{\mathbf{0}\})/ \mathbb{C} ^*$ and $\mathbf{s}$ a
lift to $\textrm{M}_{3 \times 3}(\mathbb{C})\setminus \{\mathbf{0}\}$
of $S$. The matrix $\mathbf{s}$ induces a non-zero linear
transformation $s:\mathbb {C}^{3}\rightarrow \mathbb {C}^{3}$, which
is not necessarily invertible. Let $\textrm{Ker}(s) \subsetneq
\mathbb{C} ^3$ be its kernel and let $\textrm{Ker}(S)$ denote its
projectivization to $\mathbb{P} _{\mathbb{C}} ^2$, taking into account
that $\textrm{Ker}(S):= \varnothing$ whenever
$\textrm{Ker}(s)=\{(0,0,0)\}$.
\subsection{Discontinuous actions on $\Proy ^2$}

\begin{definition}Let $G \subset \textrm{PSL}(3, \mathbb{C})$ be a
 group. We say that $G$ is a \emph{complex Kleinian group} if it acts properly and
discontinuously on an open non-empty $G$-invariant set $U \subset
\Proy ^2$, that means, for each pair of compact subsets $C,D
\subset U$, the set
$$\{g \in G : g (C) \cap D \ne \varnothing\}$$
is \emph{finite}.
\end{definition}

One of the main difficulties in deciding whether a group G is Kleinian complex is to find an open set verifying the definition above. 
In order to give an answer to this problem we propose study  two  mathematical concepts: The Equicontinuity set of $G$ and the Kulkarni discontinuity 
region of $G$. Now, we discusses each one of these definitions.

\subsection{The equicontinuity set} \label{sectionofquicontinuity}

The concept of equicontinuity has long been studied in mathematics. For convenience to reader, we include the definition and notation that
we use in this work.
\begin{definition}
The \emph{equicontinuity set} for a family $\mathcal{F}$ of
endomorphisms of $\Proy ^2$, denoted $\textrm{Eq}(\mathcal{F})$ is
defined as the set of points $z \in \Proy ^2$ for which there is an
open neighbourhood $U$ of $z$ such that $\{f\big| _U : f \in
\mathcal{F}\}$ is a normal family.
\end{definition}

This modern approach and ideas of this concept were studied by Angel Cano  in his Ph.D thesis, however  thanks to reference of Ravi Kulkarni about 
Myrberg works, we found that some of these results  have already been discovered before, only in an arcane mathematical language; however it is fair to 
recognize  Angel Cano for  rediscovering these results and apply with success to  theory of complex Kleinian groups.
\begin{definition} Let $G \subset \textrm{PSL}(3, \mathbb{C})$ be
a discrete group. If
$$G'=\{ S \textrm{ is a pseudo-projective map of } \Proy ^2 :
S \textrm{ is a cluster point of }  G\}; $$
then the \emph{Myrberg limit set} (see \cite{Myr}) is defined as the
set
$$\Lambda _{\textrm{Myr}}(G)=\bigcup _{S \in G '} \textrm{Ker}(S).$$
\end{definition}

Myrberg \cite{Myr} shows that $G$ acts properly and discontinuously
on $\Proy ^2 \setminus \lmyr(G)$.

\begin{theorem}\label{tooltheorem} (See \cite{BCN})
If $G \subset \textrm{PSL}(3, \mathbb{C})$ is a discrete group, then:
\begin{enumerate}
\item The group $G$ acts properly and discontinuously on
$\textrm{Eq}(G)$.
\item The equicontinuity set of $G$ satisfies:
$$\textrm{Eq}(G)=\Proy ^2 \setminus \Lambda _{\textrm{Myr}}(G)$$
\item If $U$ is an open $G$-invariant subset such that $\Proy ^2 \setminus
U$ contains at least three complex lines in general position, then
$U \subset \textrm{Eq}(G)$.
\end{enumerate}
\end{theorem}
\subsection{Kulkarni discontinuity region} \label{sectionofKuldisreg}

In 1978 Ravi Kulkarni motivated by the  study of classical theory of Kleinian groups defined a limit set
for groups of homeomorphism acting on locally compact Hausdorff space. For  convenience to reader, we explain this  construction in the context of
projective spaces
\begin{definition}
If $G \subset \textrm{PSL}(3, \mathbb{C})$ is a group, Kulkarni defines
(see \cite{Kul}):
\begin{itemize}
\item The set $L_0(G)$ as the closure of the set of points in $\Proy
^2$ with infinite isotropy group.

\item The set $L_1(G)$ is the closure of the set of cluster points of
the orbit $G z$, where $z$ runs over $\Proy ^2 \setminus L_0(G)$.

\item The set $L_2(G)$ is the closure of the set of cluster points of
the family of compact sets $\{g (K) : g \in G\} $, where $K$ runs
over all the compact subsets of $\Proy ^2 \setminus (L_0(G) \cup
L_1(G))$.
\end{itemize}

The \emph{Kulkarni limit set} of $G$ is defined as
$$\lkul (G)=L_0(G) \cup L_1(G) \cup L_2(G).$$

The \emph{Kulkarni discontinuity region} of $G$ is defined as:
$$\Omega (G) = \Proy ^2 \setminus \lkul (G).$$
\end{definition}

Kulkarni proves in \cite{Kul} that $G$ acts properly and
discontinuously on the set $\Omega(G)$. However, $\Omega(G)$ is not
necessarily the maximal open subset of $\Proy ^2$ where $G$ acts
properly and discontinuously. 

We notice that the definition of Kulkarni's limit set is a generalization of  limit set of hyperbolic geometry. A difficulties of Kukarni's approach it is very
hard to give an explicit computations of these limit set. In the P.H.D thesis of Juan Navarrete (\cite{Na0}, \cite{Na}) we can find these computations for the cyclic
subgroups of $\textrm{PSL}(3,\mathbb{C})$ and also for discrete subgroups of $\textrm{PU}(2,1)$.


A third possibility to see if a group is a complex Kleinian group consist  to postulate the existence of an open maximal 
where the group acts properly and discontinuously, but this approach does not say how to build this  $G$-invariant open set, 
however when we ensure their existence this  open set has good properties \cite{BCN5}.

\subsection{Four complex Kleinian groups}
This section is devoted of complex Kleinian groups of $\textrm{PSL}(3,\mathbb{C})$ with the maximum numbers of complex projective lines contained 
in its Kulkarni's
limit set is four. For simplicity  we called this kind of  group \emph{four complex Kleinian groups}. In \cite{BCN4}, the authors gave a 
caractherization
of  four complex Kleinian groups. Because in the present work, the article \cite{BCN4} is essential,  we reproduce briefly the main ideas 
and the notation used.

Let  $A\in \textrm{SL}(2,\mathbb{Z})$, with $|tr(A)|>2$, we define the  following discrete  subgroup of $\textrm{PSL}(3,\mathbb{C})$, 
called  hyperbolic toral group

\[
 G_A=
\left \{ \left (
\begin{array}{lll}
A^k  & \mathbf{b} \\
 \mathbf{0} & 1\\
\end{array}
\right ) \, \Big| \,\,  \mathbf{b}\in M(2\times 1,\mathbb{Z}), \, k\in
\mathbb{Z} \right \},
\]

The group $G_A$ is a four complex Kleinian group and moreover if $G$ is a four complex Kleinian groups, then there exist $G_A$ such that
$[G:G_A]\leq 8$.

It is possible to conjugate the  group $G_A$ to a group,  still denoted by $G_A$, where  each element is of the form

\[\left (
\begin{array}{lll}
\lambda^k & 0 & ny_0+mx_0 \\
 0 & \lambda^{-k}& nx_0 + mz_0\\
 0&0&1
\end{array}
\right )
\]
where $k$,$n$ and $m$ run over $\mathbb{Z}$ and $\lambda$ is  one of the eigenvalues of $A$. 
At this point it is not hard to see the Kulkarni discontinuity region consist of  four  disjoint copies of $\mathbb{H}^{\pm}\times \mathbb{H}^{\pm}$, 
where $\mathbb{H}^+$ is the  upper half plane and $\mathbb{H}^{-}$ is the  lower half plane.

\subsection{The $\sol$ geometries}
 $\sol$  is one of the eight geometries defined por William Thurston in his famous program of geometrization  of compact three manifolds. We define  
 the group $\sol$ as follows: given that  the space $\mathbb{R}^2\times \mathbb{R}$ we define the  group operation:

$$\left(
\begin{array}{cc}

\left(
\begin{array}{c}
x_1\\
y_1\\
\end{array}
\right),t_1

\end{array}
\right)
\cdot 
\left(
\begin{array}{cc}

\left(
\begin{array}{c}
x_2\\
y_2\\
\end{array}
\right),t_2

\end{array}
\right)
= 
\left(
\begin{array}{cc}

\left(
\begin{array}{c}
x_1+e^{t_1}x_2\\
y_1+e^{-t_1}y_2\\
\end{array}
\right),t_1+t_2

\end{array}
\right)
$$

With the previous operation, we have that $\mathbb{R}^3$ is a Lie group and it is denoted by $\sol$. It is well know that 
$\sol$ is  3-Riemannian manifold
with metric  $ds^2= e^{2t}dx^2 + e^{-2t}dy^2 +dt^2$, when the group $\sol$  is the isometry group. An  interesting fact about the $\sol$ geometries, 
is given by the following theorem of \cite{H}, that we state for convenience:

\begin{proposition}
 Let $A$, $B$ in $\textrm{GL}(2,\mathbb{Z})$ be two matrices with traces of absolute value strictly larger than 2. The semi-direct product 
 $\mathbb{Z}^2\rtimes_{A}\mathbb{Z}$ and $\mathbb{Z}^2\rtimes_{B} \mathbb{Z}$ are
 
 \begin{itemize}
  \item [a)] isomorphic if and only if $A$ is conjugate in $\textrm{GL}(2,\mathbb{Z})$ to $B$ or $B^{-1}$.
  \item [b)] quasi-isometric in all cases.
 \end{itemize}

\end{proposition}
Then  $\sol/(\mathbb{Z}^2\rtimes_{A}\mathbb{Z})$ gives examples of compact three manifolds where the topological type is determined by 
the fundamental group.For more details about this subject the reader we can  see \cite{Sc},\cite{Th} and \cite{H}.	
	
	\section{Foliation of $\polyDisc$ by $\sol$}
	
	\begin{definition}
		Given $\lambda > 0$, and $\left(\begin{smallmatrix}
			a & b \\ c & d
		\end{smallmatrix}\right)$ a nondegenerated real matrix. $\sol$ is the 
		group consisting of all matrices 
		 of the form
		\[ \solRep, \]
		with $t, x, y \in\R$ arbitrary.
	\end{definition}
	
	\begin{theorem}
		$\sol$ is isomorphic to the semidirect product $ \R^2\rtimes \R$
	\end{theorem}
	
	\begin{proof}
		See \cite{H} for a proof.
	\end{proof}
	
	\begin{definition}
		Let $z_1, z_2 \in \Hbb$. We define the action of $\sol$ in 
		$\Hbb\times\Hbb$ by
		\[ \solRep \, \begin{pmatrix}
			z_1 \\ 
			z_2 \\
			1
		\end{pmatrix} = \begin{pmatrix}
		 \lambda^t\, z_1 + a\,x + b\,y\\
		 \lambda^{-t}\, z_2 + c\,x + d\, y\\
		 1
		\end{pmatrix}. \]
	\end{definition}

	\begin{theorem}
	The natural action of the group $\sol$ on $\Hbb$ verify  the following  statements
	
	\begin{itemize}
	\item[i)] The  action  is free.
	\item[ii)] For each $z= (z_1, z_2)\in \Hbb$ the  function $f_z:\sol \rightarrow \Hbb$ defined by
	$f_z(g)= gz$ is a  smooth embedding.
	\item[iii)] Let  $e_1, e_2, e_3, e_4$ the canonical basis of $\mathbb{R}^4$, then 
	$$X= -\ln(\lambda)(\lambda^{-t}y_2e_2 + \lambda^ty_1e_4).$$
	is the normal field of  the  embedding $f_z(\sol)$ in $\Hbb$, moreover this  vector field is smooth.
	
	\end{itemize}
	
	\end{theorem}
	
	\begin{proof}
	 \begin{itemize}
	  \item[i)] Let $z = (z_1, z_2) \in \polyDisc$, and assume that $\gamma\in\sol$ is 	such that $\gamma\cdot z = z$. Let $z_k = x_k + i y_k$. Taking 
		imaginary parts in the action, we get
		\begin{align*}
			\lambda^t y_1 &= y_1, & \lambda^{-t} y_2 &= y_2,
		\end{align*}
		so, $\lambda^t = \lambda^{-t} = 1$, since each $y_k > 0$. 
	Taking real parts,
		\begin{align*}
		x_1 + a\,x + b\,y &= x_1,\\
		x_2 + c\,x + d\,y &= x_2,
		\end{align*}
		and since $\left(\begin{smallmatrix}
		a & b\\
		c & d
		\end{smallmatrix}\right)$, is non degenerated, $x = y = 0$.
		
	\item[ii)] From de definition of the action, it is clear that $f_z$ is smooth in 
		$t, x, y$, which parametrizes $\sol$.
		By straighforward computations we have that $df_z$ has 
		jacobian matrix given by,
		\[ [df_z] = \begin{pmatrix}
			\ln(\lambda)\, \lambda^t\, x_1 & a & b \\
			\ln(\lambda)\, \lambda^t\, y_1 & 0 & 0 \\
			-\ln(\lambda)\, \lambda^{-t}\, x_2 & c & d\\
			-\ln(\lambda)\, \lambda^{-t}\, y_2 & 0 & 0
		\end{pmatrix}. \]
		The vectors $(a,0,c,0)^t$, $(b, 0, d, 0)^t$ are linearly independent by 
		the condition in the coefficients $a, b, c, d$. Since $y_1 > 0$, the 
		first column in the Jacobian matrix is linearly independent with the 
		previous vectors. Therefore, $f_z$ in an immersion. 
		
		If $z_k = x_y + i\,y_k$, and $z' = (x_1', y_1', x_2', y_2')\in 
		\polyDisc$, define 
		$t$ such that $\lambda^t = y_1'/y_1$, and $(x', y')$ such that

		\begin{align*}
			\begin{pmatrix}
			 x'\\ y'
			\end{pmatrix}
			&= 
			\begin{pmatrix}
				a &  b \\
				c &  d
			\end{pmatrix}^{-1}
			\begin{pmatrix}
				x_1' - \frac{y_1'\,x_1}{y_1}\\
				x_2' - \frac{y_2'\,y_2}{y_1'}
			\end{pmatrix}		
		\end{align*} 
	
		These values for $(t, x', y')$ define a mapping $F$, from $\polyDisc$ 
		to 
		$\R^3\cong\sol$, such that $F\circ f_z = Id$. Note that $F$ is a left 
		continuous inverse for $f_z$, and hence, $f_z$ is an homeomorphism. 
	\item[iii)] Since the $e_k$ form an orthonormal basis for the euclidean metric, the 
		vector product in this basis can be calculated by the method of the 
		determinat
		\[ X = \ln(\lambda) \left|\begin{matrix}
		e_1  & e_2 & e_3 & e_4\\
		\lambda^t\, x_1 & \lambda^t\, y_1 & -\lambda^{-t}\, x_2 & 
		-\lambda^{-t}\, y_2\\
		a & 0 & c & 0\\
		b & 0 & d & 0 
		\end{matrix}		
		\right|. \]
		The result follows. Finally,let $z \in \polyDisc$ be arbitrary, then trivially $z$ belongs to the 
		leaf given by $f_z$. Since $\sol$ acts freely, $z$ is the image of the 
		identity element by $f_z$. By the expression given in theorem 
		\ref{thm:Xexpression}, 
		\[ X = -\ln(\lambda)\,\left(y_2\,e_2 + y_1\,e_4\right). \]
		Therefore, $X$ varies smoothly.
	\end{itemize}	
		
	\end{proof}

	By this theorem, if we vary $z$, we obtain a foliation $f_z(\sol)$ of 
	$\polyDisc$ by copies of $\sol$. We proceed to show that this foliation is 
	globally rectifiable. In the sequel, we will use the Euclidean metric in 
	$\polyDisc$ given by the natural identification with a subspace of $\R^4$.
	
	\begin{definition}
		For $z\in\polyDisc$ fixed, define the normal vector field to 
		$f_z(\sol)$,  $X$, by the 
		triple vector product $\partial_tf_z\wedge \partial_x f_z \wedge 
		\partial_y f_z$.
	\end{definition}


	
	
	\section{Geometry of the leaves}

	In the previous section we described how $\sol$ foliates the space
	$\polyDisc$ and gave a closed expression for a smooth vector field $X$
	normal to any leaf in the foliation. In this section, we will study the
	dynamics of the integral curves. Let $\psi(t)$ be an integral curve of
        the field, with components $\psi = (z_1(t), z_2(t))$, and
        $z_k = x_k + y_k i$ as before. From the definition of the field, it
        follows that  the integral curves
        satisfy the following set of equations:

        \begin{align*}
          \dot{x}_k &= 0,\\
          \dot{y}_k &= y_k.
        \end{align*}

        These  equations can be readily solved to get
        constant solutions in the real part of each copy of the hyperbolic
        and exponentials in the imaginary parts. The flow
        of the normal field defines a one parameter family of diffeomorphisms
        in $\polyDisc$. From now on, we will denote it by $\psi_t(z_1, z_2)$,
        where

        \[\psi_t(z_1, z_2) = (x_1, e^t y_1, x_2, e^t y_2),\]

        where the parameter $t$ is completely inextensible, i.e. $t \in
        \mathbb{R}$.

        \begin{theorem}
          The flow $\psi_t$ rules $\polyDisc$ by geodesics.
        \end{theorem}

        \begin{proof}
          Every factor
          $(x_k, e^t y_k)$ in $\psi_t$ corresponds to the parametrization
          of a vertical geodesic in $\Hbb$ with the hyperbolic metric.
          Since the metric we consider is homotetic to the standard hyperbolic
          metric, with a constant factor, $(x_k, e^t y_k)$ will be the
          parametrization of a geodesic with this metric as well. Since the
          metric in $\polyDisc$ is a product, the result follows.
        \end{proof}

        \begin{theorem}
          The action of $f_z$ is \emph{equivariant} with the action of
          the one parameter family given by the flow, that is,
          \[\psi_s\circ f_z = f_{\psi_s(z)}.\]
        \end{theorem}

        \begin{proof}
          \[\psi_s\circ f_z (t, x, y) = \left(\lambda^tx_1 + ax + by,
            \lambda^t e^s y_1, \lambda^{-t}x_2 + cx + dy,
            \lambda^{-t} e^s y_2\right),\]
          which is the same expression obtained calculating $f_{\psi_s(z)}$.
        \end{proof}

        \begin{remark}
          If we pull back $f_z$ with the mapping $\Phi^{-1}$ defined previously,
          the expression in the theorem aquires the
          simpler and equivalent form
          \begin{align}\label{eq:leaves-equivariance}
            \psi_s\circ f_z (t, x, y) = \left(e^tx_1 + x,
            e^{t + s} y_1, e^{-t}x_2 + y,
            e^{-t + s} y_2\right).
          \end{align}
        \end{remark}

        \begin{theorem}\label{thm:sol-metrics}
          Let $z = (y_1\,i, y_2\,i)$. The induced metric in the leaf
          $f_z\circ\Phi^{-1}(\sol)$
          is
          \[dt^2 + \frac{e^{-2t}}{2y_1^2} dx^2 + \frac{e^{2t}}{2y_2^2} dy^2.\]
          In particular, if $y_1 = y_2 = 1/\sqrt{2}$, $f_z$ is an isommetric
          embedding of $\sol$ into $\polyDisc$.
        \end{theorem}

        \begin{proof}
          Given the definition of $\Phi$,
          \[f_z\circ\Phi^{-1}(t, x, y) = (x, e^ty_1, y, e^{-t}y_2).\]
          Therefore, the jacobian matrix of $f_z\circ\Phi^{-1}$ is
          \[
            \begin{pmatrix}
              0 & 1 & 0\\
              e^ty_1 & 0 & 0\\
              0 & 0 & 1\\
              -e^{-t}y_2 & 0 & 0
            \end{pmatrix}.
          \]
          Applying the product metric to the basis vectors
          $e^ty_1 e_2 - e^{-t}y2 e_4$, $e_1$, $e_3$ we arrive to the result.
        \end{proof}

        In the sequel, unless otherwise stablished, $z_0$ will denote the
        special point $1/\sqrt{2}\, (i, i)$.

        \begin{corollary}\label{cor:rigid-leaves}
          The leaves $f_{\psi_s(z_0)}(\sol)$ can be
          identified with $\sol$, up to a homotecy in the direction spanned by
          $xy$ coordinates.
        \end{corollary}

        \begin{proof}
          By virtue of equation \ref{eq:leaves-equivariance}, we have
          \[\psi_s\circ f_{z_0}(t, x, y) = \left(x, \frac{1}{\sqrt{2}} e^{t+s},
            y, \frac{1}{\sqrt{2}} e^{-t+s}\right).\]
        Procceeding as in the proof of theorem \ref{thm:sol-metrics}, if we
        pull back the induced metric to $\sol$, we get
        \[dt^2 + e^{-2(t + s)} dx^2 + e^{2(t - s)}dy^2.\]
        By a simple algebraic manipulation, the previous metric is
        \[dt^2 + e^{-2t}e^{-2s}dx^2 + e^{2t}e^{-2s}dy^2.\]

        Define $F_s: \sol \to \sol$ by $F_s(t, x, y) = (t, e^s x, e^s y)$.
        Another pullback, this time with $F_s$, turns the metric into
        \[dt^2 + e^{-2t} dx^2 + e^{2t} dy^2.\]
      \end{proof}
      \begin{theorem}
        The foliation is globally rectifiable: there is a diffeomorphism
        $\Psi:\R^3\times\R \to  \polyDisc$, such that each
        hyperplane $\R^3\times\{c\}$ is diffeomorphic to a leaf.
      \end{theorem}

      \begin{proof}
        $\sol$ is diffeomorphic to $\R^3$ in a natural way. Any $\gamma\in\sol$
        is uniquely determined by a triplet $(t, x, y)$. Define
        $\Psi:\R^3\times\R \to \polyDisc$ by
        $\Psi(t,x,y, s) = \psi_s\circ f_{z_0}(\gamma)$. $\Psi$ is injective
        because the
        action is free. Given
        $z'\in\polyDisc$, there is a leaf going through it, and since
        $\psi_s(z_0)$ traverses all the leaves, there is $s$,
        such that $\psi_{-s}(z')$ is in the leave passing through $z_0$. Let
        $\gamma \in \sol$ be such that $\psi_{-s}(z') =
        f_{z_0}(\gamma)$. Therefore,
        \[z' = \psi_s\circ f_{z_0}(\gamma),\]
        which implies that $\Psi$ is also surjective.
        Finnally, a direct calculation yields $\Psi$'s jacobian to be
        \[ [d\Psi] = \begin{pmatrix}
            0                 &  1 &  0 & 0\\
            e^{t + s}/\sqrt{2} &  0  &  0 & e^{t + s}/\sqrt{2}\\
            0                 & 0  &  1 & 0\\
            -e^{-t + s}/\sqrt{2} & 0  &  0 & e^{-t + s}/\sqrt{2}
          \end{pmatrix},\]
        which is nondegenerated in the whole domain. By the inverse function
        theorem,
        $\Psi$ is a diffeomorphism. The last claim follows from the fact that
        $\psi_s$ maps leaves onto leaves.
      \end{proof}

      \begin{theorem}
        The previous diffeomorphism can be modified, such that not only maps
        the foliation to a cartesian product globally, but also maps each leave
        in the foliation isommetrically to $\sol$.
      \end{theorem}

      \begin{proof}
        By pulling back the metric in $\polyDisc$ with the previous
        diffeomorphism, we arrive to
        \[dt^2 + e^{-2(t + s)} dx^2 + e^{2(t - s)}dy^2 + ds^2,\]
        which is analogous to the expression in corollary
        \ref{cor:rigid-leaves}. Let
        \[\tilde{\Psi}(t, x, y, s) = \Psi(t, e^sx, e^sy, s).\]
        $\tilde{\Psi}$ is a leave preserving diffeomorphism such that, for
        fixed $s$, maps isommetrically $\sol$ into the leave $\R^3\times\{s\}$.
      \end{proof}

      \begin{remark}
        It doesn't look like it could be posible to find an even more rigid
        difeommorphism, i.e. a global isommetry $\sol\times \R \cong \polyDisc$
        preserving leaves. However, a proof is pending in this regard.
      \end{remark}

      \section{Extrinsic geometry}

      Recall the metric we endowed in $\polyDisc$ is homotetic to the product
      of hyperbolic metrics:

      \[\frac{dx_1^2 + dy_1^2}{2y_1^2} + \frac{dx_2^2 + dy_2^2}{2y_2^2}.\]

      \begin{theorem}
        Integral curves of the normal field $X$ are geodesics.
      \end{theorem}

      \begin{proof}
        We had previsouly found that $X$'s integral curves are given by
        $\gamma(t) = (x_1, e^ty_1, x_2, e^ty_2)$. Let $\phi(t)$ be a smooth
        curve in $\Hbb$ with the homotetic metric. Then,
        \[\left|\left|\dot{\phi}(t)\right|\right|^2 = \frac{\dot{x}^2 +
            \dot{y}^2}{2 y^2},\]
        which is half the standard hyperbolic square length. Therefore, a curve
        minimizes hyperbolic arc lenght if and only if minimizes the homotetic
        metric arch length, i.e. geodesics in both cases are the same. It is
        a well known fact that the vertical curves $(x_k, e^ty_k)$ are geodesics
        in hyperbolic space. Finnally, since $\gamma$ can be projected in
        two geodesics and the metric is a product, $\gamma$ is a geodesic in
        $\polyDisc$. (see 3.15 in \cite{gallot2012riemannian}).
      \end{proof}

      \begin{theorem}
        There are isommetries in $\polyDisc$ acting transitively and sending
        leaves onto leaves.
      \end{theorem}

      \begin{proof}
        We work in the $\R^3\times \R$ picture with the $\tilde{\Psi}$
        isommetry. A straightforward calculation shows that the mappings
        \[(t, x, y, s) \longmapsto (t + t', e^{t' + s'}x + x',
          e^{-t' + s'}y + y', s + s')\]
        are isommetries. The first claim comes from the fact that given a pair
        of points $(t_k, x_k, y_k, s_k)$, there
        exists exactly one such isommetry sending one onto another. That this
        isommetry sends
        leaves onto leaves is obvious, since under this diffeomorphism, they
        correspond to hypersurfaces $s = constant$.
      \end{proof}

      We aim to calculate the distance between any pair of leaves. Recall in
      any metric space, the distance from a point $p$ to a set
      $S \neq \emptyset$ is given by the expression
      \[d(p, S) = \inf \{d(p, x) : x \in S \}.\]


      \begin{theorem}
        The separation between two leaves in $\polyDisc$ is constant. Moreover,
        if leaves are parametrized with the normal field affine parameter,
        then leaves separation is given by the difference $|s - s'|$ between
        the parameter corresponding to any leaf.
      \end{theorem}

      \begin{proof}
        A point in a leaf can be parametrized as
        \[\left(x, \frac{e^{s + t}}{\sqrt{2}}, y, \frac{e^{s - t}}{\sqrt{2}}
          \right),\]
        where $x, y ,t$ are arebitrary, and $s$ is the parameter corresponding
        to the leaf. Given a second point in another leaf, say
        $\left(x',{e^{s' + t'}}/{\sqrt{2}}, y', {e^{s' - t'}}/{\sqrt{2}}
        \right)$, and since the metric is a product, we can find a geodesic
        minimizing arc lenght in $\polyDisc$, such that, in each factor $\Hbb$,
        the distance is also minimized. On the other hand, the metric we use
        in each factor of $\polyDisc$ is half the hyperbolic distance, for
        wich a well known formula gives us the distance. Let $\rho_k$
        denote the distance in each factor with our metric, then,
        \begin{align*}
          \cosh\left(\sqrt{2}\rho_1\right) &= 1 + \frac{2 \left(x - x'\right)^2
                          + \left(e^{s + t} - e^{s' + t'}\right)^2}
                          {2 e^{s + t}e^{s' + t'} },\\
          \cosh\left(\sqrt{2}\rho_2\right) &= 1 + \frac{2\left(y - y'\right)^2
                          + \left(e^{s - t} - e^{s' - t'}\right)^2}
                          {2 e^{s - t}e^{s' - t'} },\\
        \end{align*}
        where the $\sqrt{2}$ factor within the hyperbolic cosine is due to the
        factor relating standard hyperbolic metric with ours. The previous
        expression shows that, in order to get the minimum
        distance, $x'$ must be equal to $x$ and $y'$ to $y$. Simplifying the
        previous expresions for such values of $x'$ and $y'$, we find

        \begin{align*}
          \cosh\left(\sqrt{2}\rho_1\right) &= \cosh\left(s - s' + t - t'\right),\\
          \cosh\left(\sqrt{2}\rho_2\right) &= \cosh\left(s - s' + t' - t\right).
        \end{align*}

        Therefore,

        \begin{align*}
          \rho_1 &= \frac{|s - s' + t - t'|}{\sqrt{2}},  &
          \rho_2 &= \frac{|s - s' + t' - t|}{\sqrt{2}},
        \end{align*}

        and the distance in the product metric is given by
        $\sqrt{\rho_1^2 + \rho_2^2}$. In order for this distance to be a
        minimum, a short analysis shows that one must take $t' = t$,
        and the theorem statement follows.
      \end{proof}

      \begin{theorem}
        The principal curvatures of each leaf are $-1$ with multiplicity two,
        and $0$. The principal directions are determined by the integral
        curves of the vectors $\partial_x$, $\partial_y$, $\partial_t$
        respectively.
      \end{theorem}

      \begin{proof}
        Recall the principal curvatures and directions for an orientable
        submanifold $M$
        are
        determined by the \emph{shape operator}, $S$, which in codimension one,
        can be regarded as the mapping $TM \to TM$ given by
        $v_x \mapsto \nabla_{v_x}X$, where $X$ is the normal field to the
        manifold,
        compatible with orientation. Here, the principal directions and
        curvatures are the shape operator eigenvectors, and eigenvalues.
        Consider a leaf embedded in $\polyDisc$,
        \[\left(x, \frac{e^{-t-s}}{\sqrt{2}}, y, \frac{e^{t-s}}{\sqrt{2}}\right),
        \]
        with normal field $X = x_2\partial_2 + x_4\partial_4$, where
        $x_2 = {e^{-t-s}}/{\sqrt{2}}$, $x_4 = {e^{t - s}}/{\sqrt{2}}$. A
        straightforward calculation shows that $\nabla X =
        -dx_1\otimes\partial_1 - dx_3 \otimes\partial_3$, i.e., the shape
        operator is diagonal, once expressed in the base for the tangent space
        to the leaf,
        spanned by the coordinate vectors $\partial_1$, $\partial_3$, and the
        vector $-x_2\partial_2 + x_4\partial_4$. Moreover, the eigenvalues are
        precisely $\{-1, -1, 0\}$.
      \end{proof}

      \section{The Heisenberg group}

      Given a symplectic vector space, $V$, with symplectic form $\omega$. Recall
      the \emph{Heisenberg} group, $\heisenberg$, is the space $V\times\R$,
      with the
      product operation given by
      \[(v, t)*(w, s) = (v + w, t + s + \omega(v, w)).\]

      If $V$ is of dimension 2, and $\{\partial_p, \partial_q\}$ is a
      symplectic base for $V$,
      that is, $\omega(\partial_p, \partial_q) = 1$,
      a well known fact from Lie group theory is that there is a faithfull
      representation $\heisenberg \to \mathrm{SL(3,\R)}$ given by

      \[\l(p \partial_p +  q \partial_q t \r) \longrightarrow
        \begin{pmatrix}
          1 & p & t + \frac{1}{2}pq\\
          0 & 1 & q\\
          0 & 0 & 1
        \end{pmatrix}.
      \]

      We will use this representation and identify $\heisenberg$ with
      $\mathrm{SL(3, \R)}$. Therefore, we will identify $\heisenberg$ with
      $\R^3$, with group structure,

      \[(a, b, c) * (a', b', c') = (a + a', b + b', c + c' + a b'),\]

      which corresponds to the matrix product

      \[  \begin{pmatrix}
          1 & a & c\\
          0 & 1 & b\\
          0 & 0 & 1
        \end{pmatrix} \begin{pmatrix}
          1 & a' & c'\\
          0 & 1 & b'\\
          0 & 0 & 1
        \end{pmatrix}.
      \]

      With these identifications, there is a natural action
      $\heisenberg \circlearrowleft \C\times\Hbb$:

      \[
        \begin{pmatrix}
          1 & a & c\\
          0 & 1 & b\\
          0 & 0 & 1
        \end{pmatrix}
        \begin{pmatrix}
          z \\ w \\ 1
        \end{pmatrix} =
        \begin{pmatrix}
          z + a w + c \\ w + b \\ 1
        \end{pmatrix},
      \]

      which we will denote $(a, b, c) * (z, w)$.

      \begin{theorem}
        The action of $\heisenberg$ in $\C\times\Hbb$ is free.
      \end{theorem}

      \begin{proof}
        If $(a, b, c) * (z, w) = (z, w)$, then
        \begin{align*}
          z + aw + c &= z,\\
               w + b &= w.
        \end{align*}
        From this linear system, one deduces that $a = b = c = 0$.
      \end{proof}

      \begin{theorem}
        For fixed $(z, w)\in\C\times\Hbb$, the orbit $h\in\heisenberg
        \mapsto h*(z, w)$,
        defines a differentiable embedding $\heisenberg \hookrightarrow
        \C\times\Hbb$.
      \end{theorem}

      \begin{proof}
        The map is injective, since the action is free. Let $w = p + q i$, the
        jacobian matrix of the mapping
        in $(a, b, c) \in\heisenberg$ is given by
        \[
          \begin{pmatrix}
            p & 0 & 1\\
            q & 0 & 0\\
            0   & 1 & 0\\
            0   & 0 & 0
          \end{pmatrix}.
        \]
        Since the jacobian has rank 3, the action defines a local
        diffeomorphism, hence an embedding.
      \end{proof}

      Therefore, the action of $\heisenberg$ defines a foliation of
      $\C\times\Hbb$, in analogy with the foliation of $\polyDisc$ generated
      by $\sol$.

      \begin{theorem}
        Consider $\C\times\Hbb$ as a subset of $\R^4$, but with the
        product metric of the euclidean metric in $\C$ and the hyperbolic
        metric in $\Hbb$.
        If $e_1, \ldots, e_4$ denotes the canonical coordinates in $\R^4$,
        and $(p, q)$ denote the coordinates in $\Hbb$, then
        the vector field $X = q e_4$ is unitary and perpedicular to any leaf
        of the foliation generated by $\heisenberg$.
      \end{theorem}

      \begin{proof}
        From the previous theorem, the vector fields $p e_1 + q e_2$, $e_3$,
        $e_1$ generate the tangent space to the orbit of
        $(z, w)\in\C\times\Hbb$, where $w = p + q i$. Since the metric is a
        product, $X$ is perpedicular to $pe_1 + q e_2$ and $e_1$. Moreover,
        the metric in $\Hbb$ is conformal to the euclidean, and therefore
        $X$ is perpedicular to $e_3$. Finnally, $q e_4$ is unitary in the
        hyperbolic metric.
      \end{proof}

      \begin{theorem}
        Let $(z, w) \in \C\times\Hbb$, $z = x + y i$, $w = p + q i$. The
        integral curves of $X$ are geodesics.
      \end{theorem}

      \begin{proof}

      \end{proof}

      Altough in this case, the action induced by the normal field $X$
      \emph{is not} equivariant, we can describe in a precise way the
      quotients $\C\times\Hbb/\Gamma$, where $\Gamma$ is a discrete subgroup
      of $\heisenberg$. Moreover, if $\Gamma$ acts properly discontinous in
      $\C\times\Hbb$, it has to act in the same way in Heisenberg, because
      the slices $\heisenberg\times\{qi\}$ are preserved. This is a
      general property of lie groups.

      \begin{theorem}
        Let $X$ and $Y$ be two locally compact spaces.
        If $\Gamma \circlearrowleft X \times Y$, and the action of $g\in\Gamma$
        can be decomposed as $g\cdot(x, y) = (g\cdot x, y)$ then $\Gamma$
        acts properly discontinous in $X$ iff it acts properly discontinous in
        $X \times Y$.
      \end{theorem}

      \begin{proof}
        Let $K \subset X$ be a compact set. Fix $y \in Y$. With the product
        topology, $K\times\{y\}$ is a compact set in $X \times Y$.
        One can easily
        verify the equality
        \[\l\{g\in\Gamma: g\cdot K \cap K \neq \emptyset \r\} =
          \l\{g\in\Gamma: g\times 1 \cdot K\times\{y\} \cap K\times\{y\}
          \neq \emptyset \r\}.\]
        If $\Gamma$ acts properly discontinous in $X \times Y$, the previous
        equality implies that it acts properly discontinous in $X$. On the
        other hand, if $K \subset X \times Y$ is compact, the product topology
        together with the local compacity implies that we can find an open set
        $U \times V$, with $U \in X$ and $V \in Y$, such that $\bar{U}$ is
        compact in $X$, $\bar{V}$ is compact in $Y$, and
        $K \subset U\times V$. We have the contention
        \[\l\{g\in \Gamma: gK \cap K \neq \emptyset \r\} \subset
          \l\{g\in \Gamma: g\cdot\l(\bar{U}\times\bar{V}\r) \cap
          \bar{U}\times\bar{V}
          \neq \emptyset \r\}\]
        Take $g\in\Gamma$, $(x, y) \in \bar{U}\times\bar{V}$, such that
        $g\cdot(x, y) \in \bar{U}\times\bar{V}$. Since
        $g\cdot(x, y) = (g\cdot x, y)$, it follows that $g\cdot x \in \bar{U}$.
        Therefore, the second set in the previous contention is
        at the same time contained in
        \[\l\{g\in \Gamma: g\cdot\bar{U} \cap \bar{U} \neq \emptyset \r\}.\]
        If $\Gamma$ acts properly discontinous in $X$, this set
        has to be finite, and the same must be true for the set of
        intersections in $X\times Y$. i.e. $\Gamma$ acts properly
        discontinous in $X$.
      \end{proof}

      \begin{theorem}
        $\C\times\Hbb$ is diffeomorphic to $\heisenberg\times\R$, where,
        up to diffeomorphism, $\heisenberg$ acts on the first factor only.
      \end{theorem}

      \begin{proof}
        Let $\gamma = (a, b, c)\in \heisenberg$. And take
        $(0, qi)\in\C\times\Hbb$. We can describe the orbits
        $\gamma\cdot(0, qi)$:

        \[
          \begin{pmatrix}
            1 & a & c\\
            0 & 1 & b\\
            0 & 0 & 1
          \end{pmatrix} \cdot
          \begin{pmatrix}
            0 \\ qi \\ 1
          \end{pmatrix} =
          \begin{pmatrix}
            aqi + c\\ qi + b\\1
          \end{pmatrix}.
        \]

        Therefore, there is exactly one $(0, qi)$ in each orbit of the
        group action. Define $\Psi: \heisenberg\times\R\to\C\times\Hbb$ as
        \[\Psi(\gamma, q) = \gamma\cdot(0, qi).\]
        It can be show that $\Psi$ is bijective. It is a diffeomorphism,
        since an explicit calculation shows that $d\Psi$ maps the
        canonical vectors $T_{(\gamma,q)}\heisenberg\times\R\cong \R^4 \to
        T_{\gamma\cdot(0,qi)}\C\times\Hbb \cong \R^4$:

        \[\l\{\partial_1,\ldots,\partial_4\r\} \longmapsto
          \l\{q\partial_2,\partial_3, \partial_1, a\partial_2 + \partial_4\r\}.
        \]

        The last assertion follows since the action is asociative, i.e.
        $\gamma'\cdot(\gamma\cdot (0, qi)) = (\gamma'\cdot\gamma)\cdot(0, qi)$,
        and therefore, preserves the imaginary part on the second factor.
      \end{proof}

      \begin{corollary}
        If $\Gamma < \heisenberg$ is a discrete subgroup acting
        properly discontinous in $\C\times\Hbb$, up to diffeomorphism,
        $\C\times\Hbb / \Gamma \cong \heisenberg/\Gamma \times \R$, and
        the quotient $\heisenberg/\Gamma$ is a manifold, whose fundamental
        group is $\pi^1(\heisenberg/\Gamma)\cong \Gamma$.
      \end{corollary}

      \begin{example}
        Let $\heisenberg_\mathbb{Z} < \heisenberg$ be the discrete subgroup
        of Heisenberg matrices wiht integer coeficients. It can be shown that
        the unit cube $K_C = [0, 1]^{3} \subset \heisenberg$ is a fundamental
        region for the action of $\heisenberg_\mathbb{Z}$. The quotient
        $\heisenberg_\mathbb{Z}/\heisenberg$ is an example of nilmanifold,
        whose fundamental group is
        \[\heisenberg_\mathbb{Z} \cong \langle m, n, k : [m, n] = k^4 \rangle.\]
        In view of the previous results, $\heisenberg_\mathbb{Z}$ acts
        properly discontinous in $\C\times\Hbb$, and the quotient
        $\C\times\Hbb/\heisenberg_\mathbb{Z}$ is a product of a nilmanifold
        times $\R$, whose fundamental group has the previous presentation.
      \end{example}

      \subsection{Metric properties}

      Let $(p, q, t)$ denote local coordinates in $\heisenberg$. In this
      coordinates, the standard metric is defined to be
      \[dp^2 + (1 + p^2)dq^2 + dt^2 -p dq\cdot dt.\]

      However, we found no relation between the standard metric and the
      metric induced by the family of embeddedings $\heisenberg \to
      \C\times\Hbb$ induced by the action.

      \begin{theorem}
        Let $(0, y_0\cdot i)\in\C\times\Hbb$ be the base point whose orbit generates
        a diffeomorphic copy of $\heisenberg$. Then the pullback metric
        in $\heisenberg$ is
        \[y_0^2dp^2 + \frac{1}{y_0^2} dq^2 + dt^2.\]
      \end{theorem}

      \begin{proof}
        Denote by $\psi:\heisenberg \to \C\times\Hbb$ the map defined by
        the action of $\heisenberg$ in the given point, then
        \[\psi(p, q, t) = (t + p\,y_0\. i, q + y_0\,i).\]
        For fixed $h = (p, q, t) \in \heisenberg$, the derivative $d\psi_h$
        induces a linear map $T_h\heisenberg \to T_{\psi(h)}\C\times\Hbb$, such
        that the basic tangent vectors $\partial_p, \partial_q, \partial_t$ are
        send to $y_0\partial_2, \partial_3, \partial_1$ respectively.

        Upon
        identifying the basic vectors in $\heisenberg$ with its images,
        the local expression for the induced metric is obtained.
      \end{proof}

      \begin{remark}
        Note how the local expression for the metric resembles that of $\sol$.
        In fact, under the diffeomorphism $\heisenberg\times\R \to \C\times\Hbb$
        $y_0$ turns out to be an expression of the form $e^s$, for $s\in \R$.
        For $s$ fixed at least, the induced metric becomes
        \[e^{2s}dp^2 + e^{-2s} dq^2 + dt^2,\]
        which looks analogous to what would be obtained in a section of
        $\sol\times\R$.
      \end{remark}

      \begin{remark}
        Altough the diffeomorphism $\heisenberg\times\R \cong \C\times\Hbb$
        lacks a metric relation, we can calculate nevertheless the
        separation between two folitation leaves, in analogy to what we did
        with $\sol$.
      \end{remark}

      \begin{theorem}
        Let $\Psi:\heisenberg\times\R\to\C\times\Hbb$ be the diffeomorphism
        induced by the action of $\heisenberg$ in $(0, e^s\cdot i)$, for
        $s \in \R$. Then the separation between the hyperplanes (leaves)
        $\heisenberg\times\{s_0\}$ and  $\heisenberg\times\{s_1\}$ is
        precisely $|s_1 - s_0|$.
      \end{theorem}

      \begin{proof}
        Given coordinates $(p, q, t, s)$ for $\heisenberg\times\R$, the
        expression for $\Psi$ is $\Psi(p, q, t, s) = (t + p\,e^si, q + e^si)$.

        Note that the metric in $\C\times\Hbb$ is invariant under
        \emph{``horizontal traslations''}: $(p, q, t, s) \mapsto
        (p+p_0, q + q_0, t + t_0, s)$. Given a path
        $\gamma = (\gamma_1, \gamma_2):[0, 1] \to \C\times\Hbb$ connecting a
        point in the first
        leave to a point in the second. Since the metric is a product,
        the length of the second component
        $\gamma_2$, has to be a lower bound for the length of the total path:
        \[\ell(\gamma_2) \leq \int_0^1||\dot{\gamma}_2||dt
        \leq \int_0^1\sqrt{||\dot{\gamma}_1||^2 + ||\dot{\gamma}_2||^2}
        = \ell(\gamma).\]
      Consider the path $\Psi\l(0, 0, 0, e^{(s(s_1 - s_0) + s_0)}\r)$, joining
      $(0, e^{s_0}\,i)$ to $(0, e^{s_1}\,i)$. Since the projection to $\Hbb$ of
      any other path has to connect this points, the previous bound shows that
      its lenght has to be lesser than that of this path. However, in $\Hbb$,
      the minimum separation between the lines $\R\times\{\{e^{s_0}\}$ and
      $\R\times\{e^{s_1}\}$ is precisely given by the path
      $e^{(s(s_1 - s_0) + s_0)}i$, whose length is $|s_1 - s_0|$.
    \end{proof}

      \section{The topological type}

      In this paper we described two difeommorphisms:
      $\polyDisc \cong \sol\times \R$ and
      $\C\times\Hbb\cong\heisenberg\times\R$. In the first case, the action
      of the group was equivariant with respect to the flow of the normal
      field to the orbits. In the second case, there isn't such equivariance.
      However, in both cases, the action of the group can be
      \emph{factored out}, that is to say, in both cases, if $G$ denotes
      the corresponding group and $X$ the target space, there is
      a diffeomorphism,
      \[X \cong G \times \R,\]
      such that the action of $G$ preserves the leaves $G\times\l\{t\r\}.$
      Moreover, up to diffeomorphism, the action can be described as
      $\gamma'\cdot(\gamma, x) = (\gamma'\gamma, x)$.

      Let $\Gamma < G$ be a discrete subgroup. The previous discusion shows
      that we can describre the quotient:
      $\Gamma/X \cong (\Gamma/G) \times \R$. In particular, since the second
      factor is contractible, we must have isomorphisms
      $\pi_1(\Gamma/X) \cong \pi_1(\Gamma/G)$, so that we can describe both
      quotients $\Gamma/(\polyDisc)$, and $\Gamma\times(\C\times\Hbb)$, where,
      for an abuse in notation, $\Gamma$ denotes distinct discrete subgroups
      of $\sol$ and $\heisenberg$.
      
      \subsection{proof of main theorem 1}
      
      Let $G$ be a complex kleinian group with maximum number of lines in general position contained in its limit set is four, then $G$
      acts properly and discontinuously in  four copies disjoints of $\polyDisc$. Without loss of generality we  assume  that $\polyDisc$ is
      $G$-invariant. By the theorem, we have
      
      $$\psi: \sol\times \mathbb{R}\rightarrow \polyDisc$$
      is a diffeomorphism $G$-equivariant, then  $\polyDisc /G$ is diffeomorphic to $(\sol/G)\times \mathbb{R}$.
      
      We notice the topological type is perfectly determinde by the group $G$.In fact,  the group $G$ is the  fundamental group of the manifold 
      $\polyDisc/G$. We remember the Kulkarni discontinuity region is equal a four copies disjoint of $\polyDisc$, Hence $\Omega/G$ is equal to 
      four disjoint copies of $\polyDisc/G$. We remark  $G$ represented a lattice of the Lie group $\sol$, then $\sol /G$ is a  compact 3 manifold.This
      last statement implies  in some sense $\sol/G$  is the compact heart of  $\polyDisc/G$.


\begin{thebibliography}{10}
\bibitem{BCN} W. Barrera, A. Cano, J. P. Navarrete, The Limit Set
of Discrete Subgroups of $\textrm{PSL} (3, \mathbb{C})$, Math. Proc.
Camb. Phil. Soc., Vol. 150, 2011, pp. 129-146.

\bibitem{BCN1} W. Barrera, A. Cano, J. P. Navarrete, One line complex Kleinian groups,
Pacific Journal of Mathematics, to appear.

\bibitem{BCN4} W. Barrera, A. Cano, J. P. Navarrete, Subgroups of $\textrm{PSL}(3,
\mathbb{C})$ with four lines in general position in its limit set.
Conformal Geometry and Dynamics, Volume 15, Pages 160-176 (October
11, 2011).

\bibitem{BCN5} W. Barrera, A. Cano, J. P. Navarrete, Pappus' theorem
and a construction of complex Kleinian groups with rich dynamics.
Bulletin of the Brazilian Mathematical Society, March 2014, Vol. 45,
issue 1, pp. 25-42.
\bibitem{BCN6} W. Barrera, A. Cano, J. P. Navarrete, On the number of lines in the limit set for 
discrete subgroups of $\textrm{PSL}(3,\mathbb{C})$. Pacific Journal of Mathematics, 
Vol. 281, No.1, 2016, pp. 17-49.
\bibitem{CNS} A. Cano, J. P. Navarrete, J. Seade {\it Complex Kleinian
Groups}, Progress in Mathematics, 303, Birk\"auser 2013.

\bibitem{cps} A. Cano, J. Parker, J. Seade, Actions of
$\mathbb{R}$-Fuchsian groups on $\mathbb{C}\mathbb{P}^2$, preprint.

\bibitem{cs0}
A. Cano, J. Seade; On the Equicontinuity region of Discrete
Subgroups of PU(1,n). J Geom. Anal. (2010) 20, pp 291-305.

\bibitem{CS} A. Cano, J. Seade, On Discrete groups of Automorphisms
of $\mathbb{C}\mathbb{P}^2$. Geometriae Dedicata. Published on line 08 January
2013. DOI 10.1007/s10711-012-9816-2

\bibitem{H}P. de la Harpe, Topics in Geometric Group Theory, Chicago Lectures in Mathematics, 2000.





\bibitem{Kul} R. S. Kulkarni,  \textit{Groups with Domains of Discontinuity},
Math. Ann. No 237, pp. 253-272 (1978).


\bibitem{Ma} B. Maskit, Kleinian Groups, A series in comprehensive
studies in mathematics. 287.

\bibitem{Myr} P. J. Myrberg, Untersuchungen über automorphen
Funktionen beliebig vieler Variabeln. Acta Math 46, 215-336, 1925

\bibitem{Na0} J.P. Navarrete, On the limit set of discrete subgroups of $PU(2,1)$.
Geometriae Dedicata, 122, 2006, pp. 1-13.

\bibitem{Na} J.P. Navarrete, The Trace Function and Complex Kleinian
Groups. International Journal of Mathematics. Volume 19, No. 7,
August 2008, pp. 865-890.



\bibitem{Sc} P.Scoot, The geometries of 3-manifolds, Bull. London Math.Soc.,15 (1983), 401-487.
\bibitem{SV} J. Seade, A. Verjovsky, Actions of Discrete Groups on
Complex Projective Spaces. Contemporary Math., 269, 2001, pp.
155-178.
\bibitem{Th}. W.Thurston, Three-Dimensional Geometry and Topology, edited by Silvio Levy. Princeton University Press, 1997.


\end{thebibliography}
\end{document}